\documentclass[12pt]{article}
\usepackage{fullpage}
\usepackage{amssymb,amsmath,amsfonts,amsthm,latexsym,enumerate,url,cases}
\numberwithin{equation}{section}
\usepackage{hyperref}
\usepackage{indentfirst}
\hypersetup{colorlinks=true,citecolor=blue,linkcolor=blue,urlcolor=blue}

\newtheorem{theorem}{Theorem}[section] %
\newtheorem{lemma}[theorem]{Lemma} %

\newtheorem*{Pro}{Problem 1}
\newtheorem*{theoremA}{Theorem A}

\usepackage{cite}

\begin{document}
\title{The lower bound of weighted representation function }

\author{  Shi-Qiang Chen\footnote{ E-mail: csq20180327@163.com (S.-Q. Chen).} \\
\small  School of Mathematics and Statistics, \\
\small  Anhui Normal University,  Wuhu  241002,  P. R. China\\
}
\date{}
\maketitle \baselineskip 18pt \maketitle \baselineskip 18pt

{\bf Abstract.} For any given set $A$ of nonnegative integers and for any given two positive integers $k_1,k_2$, $R_{k_1,k_2}(A,n)$
 is defined as the number of solutions of the equation $n=k_1a_1+k_2a_2$ with $a_1,a_2\in A$. In this paper, we prove that if integer $k\geq2$ and set $A\subseteq\mathbb{N}$ such that $R_{1,k}(A,n)=R_{1,k}(\mathbb{N}\setminus A,n)$ holds for all integers $n\geq n_0$, then $R_{1,k}(A,n)\gg \log n$.


\vskip 5mm

\section{Introduction}

Let $\mathbb{N}$ be the set of all nonnegative integers. For a given set $A\subseteq\mathbb{N}$, $n\in \mathbb{N}$, representation functions $R_1(A,n)$, $R_2(A,n)$ and $R_3(A,n)$ are defined as
$$R_1(A,n)=\mid\{(a,a'):n=a+a',~a,a'\in A\}\mid,$$
$$R_2(A,n)=\mid\{(a,a'):n=a+a',~a<a',~a,a'\in A\}\mid,$$
$$R_3(A,n)=\mid\{(a,a'):n=a+a',~a\leq a',~a,a'\in A\}\mid,$$
respectively. S\'{a}rk\"{o}zy
once asked the following question$:$ for $i\in\{1,2,3\}$, are there two sets of nonnegative integers $A$ and $B$ such that
$$|(A\cup B)\setminus(A\cap B)|=+\infty,$$
$R_i(A,n)=R_i(B,n) $ for all sufficiently large integers $n$?
This problem of S\'{a}rk\"{o}zy has been solved completely. Recently, many researchers have obtained many profound results around this problem of S\'{a}rk\"{o}zy. For related research, please refer to \cite{ChenLev2016}-\cite{Qu2014}, \cite{RS2017}-\cite{Tang2016} .

For any given two positive integers $k_1,k_2$ and set $A\subseteq \mathbb{N}$, weighted representation function $R_{k_1,k_2}(A,n)$ is defined as the number of solutions of the equation $n=k_1a_1+k_2a_2$ with $a_1,a_2\in A$.

In 2012, Yang and Chen \cite{YangChen2012'} studied weighted representation function. They proved that if $k_1$ and $k_2$ are two integers with $k_2>k_1\geq2$ and $(k_1,k_2)=1$, then there does not exist set $A\subseteq\mathbb{N}$ such that $R_{k_1,k_2}(A,n)=R_{k_1,k_2}(\mathbb{N}\setminus A,n)$ for all sufficiently large integers $n$, and if $k$ is an integer with $k\geq2$, then exists a set $A\subseteq\mathbb{N}$ such that $R_{1,k}(A,n)=R_{1,k}(\mathbb{N}\setminus A,n)$ for all integers $n\geq1$. They also asked the following question.
\begin{Pro} Let $k$ be an integer with $k\geq2$ and $A\subseteq\mathbb{N}$ such that $R_{1,k}(A,n)=R_{1,k}(\mathbb{N}\setminus A,n)$ for all integers $n\geq n_0$. Is it true that $R_{1,k}(A,n)\geq1$ for all sufficiently larger integers $n$? Is it true that $R_{1,k}(A,n)\rightarrow\infty$ as $n\rightarrow\infty$?
\end{Pro}
In 2016, Qu \cite{Qu2016} solved this problem affirmatively and proved that the following result.
\begin{theoremA}(See \cite[Theorem 1]{Qu2016}.) Let $k$ be an integer with $k>1$ and $A\subseteq\mathbb{N}$ such that $R_{1,k}(A,n)=R_{1,k}(\mathbb{N}\setminus A,n)$ for all integers $n\geq n_0$. Then $R_{1,k}(A,n)\rightarrow\infty$ as $n\rightarrow\infty$.
\end{theoremA}
In this paper, we continue to focus on Problem 1 and give the lower bound of weighted representation function.

\begin{theorem}\label{thm1} Let $k$ be an integer with $k\geq 2$ and $A\subseteq\mathbb{N}$ such that $R_{1,k}(A,n)=R_{1,k}(\mathbb{N}\setminus A,n)$ holds for all integers $n\geq n_0$. Then $R_{1,k}(A,n)\gg \log n$.
\end{theorem}

Throughout this paper, the characteristic function of the set $A\subseteq \mathbb{N}$ is denoted by
\begin{eqnarray*}
\chi(t)=
\begin{cases}
0          &\mbox{ $t\not\in A$},\\
1&\mbox{ $t\in A$}.\\
\end{cases}
\end{eqnarray*}
Let $C(x)$ be the set of nonnegative integers in $C$ which are less than or equal to $x$. For positive integer $k$ and sets $A,B\subseteq \mathbb{N}$, define $kA=\{ka:a\in A\}$ and $A+B=\{a+b:a\in A,~b\in B\}$.

\section{Lemmas}
\begin{lemma}(See \cite[Lemma 2]{YangChen2012'}.)\label{lem1}  Let $k\geq 2$ be an integer and $A\subseteq\mathbb{N}$. Then $R_{1,k}(A,n)=R_{1,k}(\mathbb{N}\setminus A,n)$ holds for all integers $n\geq n_0$ if and only if the following two conditions hold:

(a) for all $n_0\leq n< k+n_0$, we have
\begin{equation}\label{eq1}
\underset{a_1+ka_2=n}{\underset{a_1\geq0,a_2\geq0}{\sum}}1=\underset{a_1+ka_2=n}{\underset{a_1\geq0,a_2\geq0}{\sum}}\chi(a_1)+\underset{a_1+ka_2=n}{\underset{a_1\geq0,a_2\geq0}{\sum}}\chi(a_2);
\end{equation}
(b) for all $n\geq k+n_0$, we have
\begin{equation}\label{eq2}
\chi(n)+\chi\left(\left\lfloor \frac{n}{k}\right\rfloor\right)=1.
\end{equation}

\end{lemma}

\begin{lemma}\label{lem2}   Let $k\geq 2$ be an integer and $A\subseteq\mathbb{N}$. Then $R_{1,k}(A,n)=R_{1,k}(\mathbb{N}\setminus A,n)$ holds for all integers $n\geq n_0$, then for any $n\geq \lfloor\frac{n_0+k}{k}\rfloor+1$, we have
\begin{eqnarray}\label{eq3}
&&\chi(n)+\chi(k^in+j)=1,~~~j=0,\ldots,k^i-1, ~~~\text{if $i$ is odd};\nonumber\\
&&\chi(n)=\chi(k^in+j), ~~~~~~j=0,\ldots,k^i-1,~~~~~\text{if $i$ is even}.
\end{eqnarray}
\end{lemma}
\begin{proof}We now use induction on $i$ to prove that \eqref{eq3} is true.
By \eqref{eq2}, we have
\begin{equation}\label{eq4}
\chi(n)+\chi(kn+j)=1,~~~j=0,\ldots,k-1.
\end{equation}
Therefore,  \eqref{eq3} is true for $i=1$.

Next, we assume that \eqref{eq3} is true for $i=s$, we are going to prove the truth of \eqref{eq3} for $i=s+1$.
If $s+1$ is even, then by the induction hypothesis on $i=s$, we have
\begin{equation}\label{eq6}
\chi(n)+\chi(k^sn+j)=1,~~~j=0,\ldots,k^s-1.
\end{equation}
By \eqref{eq2}, we have
\begin{equation*}
\chi(k^sn+j)+\chi(k(k^sn+j)+u)=1,~~~j=0,\ldots,k^s-1;u=0,\ldots,k-1.
\end{equation*}
It follows from \eqref{eq6} that
\begin{equation*}
\chi(n)=\chi(k(k^sn+j)+u),~~~j=0,\ldots,k^s-1;u=0,\ldots,k-1,
\end{equation*}
that is
\begin{equation}\label{eq8}
\chi(n)=\chi(k^{s+1}n+j),~~~j=0,\ldots,k^{s+1}-1.
\end{equation}
If $s+1$ is odd, then by the induction hypothesis on $i=s$, we have
\begin{equation}\label{eq9}
\chi(n)=\chi(k^sn+j),~~~j=0,\ldots,k^s-1.
\end{equation}
By \eqref{eq2}, we have
\begin{equation*}
\chi(k^sn+j)+\chi(k(k^sn+j)+u)=1,~~~j=0,\ldots,k^s-1;u=0,\ldots,k-1,
\end{equation*}
It follows from \eqref{eq9} that
\begin{equation*}
\chi(n)+\chi(k(k^sn+j)+u)=1,~~~j=0,\ldots,k^s-1;u=0,\ldots,k-1,
\end{equation*}
that is
\begin{equation}\label{eq11}
\chi(n)+\chi(k^{s+1}n+j)=1,~~~j=0,\ldots,k^{s+1}-1.
\end{equation}
Up to now, \eqref{eq3} has been proved.

This completes the proof of Lemma \ref{lem2}.
\end{proof}

\section{Proof of Theorem \ref{thm1}}

Let $T=\lfloor\frac{n_0+k}{k}\rfloor+1$.
Given an odd $j\in[0,\lfloor\frac{\lfloor\log_{k}{\frac{n}{T}}\rfloor}{2}\rfloor]$, for any sufficiently larger integer $n$,  there exist an integer $i$  such that
\begin{equation}\label{ec1}k^i(k^{j}+1)T\leq n<k^{i+1}(k^{j}+1)T.\end{equation}
Now we are going to prove $i+j=\lfloor\log_{k}{\frac{n}{T}}\rfloor$ or $\lfloor\log_{k}{\frac{n}{T}}\rfloor-1$. In deed, if $i+j\geq\lfloor\log_{k}{\frac{n}{T}}\rfloor+1$, then $$\frac{n}{T}=k^{\log_{k}{\frac{n}{T}}}<k^{\lfloor\log_{k}{\frac{n}{T}}\rfloor+1}\leq k^{i+j}<k^{i+j}+k^i\leq\frac{n}{T},$$
 a contradiction. If $i+j\leq\lfloor\log_{k}{\frac{n}{T}}\rfloor-2$, then
 $$\frac{n}{T}< k^{i+j+1}+k^{i+1}\leq 2k^{i+j+1}\leq2 k^{\lfloor\log_{k}{\frac{n}{T}}\rfloor-1}\leq2k^{\log_{k}{\frac{n}{T}}-1}\leq k^{\log_{k}{\frac{n}{T}}}=\frac{n}{T},$$
 a contradiction.
By \eqref{ec1}, there exist $T\leq t\leq kT-1$ and $0\leq r< k^i(k^j+1)$ such that
$$n=k^i(k^{j}+1)t+r.$$
According to the value of $r$, we divide into the following two cases for discussion:
{\bf Case1.} $0\leq r\leq k^{i+j}+k^i-k-1$. Noting that $j$ is an odd, by \eqref{eq3}, we have
$$[k^{i+j-1}t,k^{i+j-1}t+k^{i+j-1}-1]\cup[k^it,k^it+k^i-1]\subseteq A~\text{or}~\mathbb{N}\setminus A.$$
Then
$$[k^i(k^j+1)t,k^i(k^j+1)t+(k^{i+j}+k^i-k-1)]\subseteq A+kA~\text{or}~(\mathbb{N}\setminus A)+k(\mathbb{N}\setminus A),$$
it follows that $n\in A+kA~\text{or}~(\mathbb{N}\setminus A)+k(\mathbb{N}\setminus A)$, which implies that
$$R_{1,k}(A,n)=R_{1,k}(\mathbb{N}\setminus A,n)\geq1.$$
Up to now, we proved that for a given odd $j\in[0,\lfloor\frac{\lfloor\log_{k}{\frac{n}{T}}\rfloor}{2}\rfloor]$, we have
\begin{equation}\label{equ1}R_{1,k}(A,n)=R_{1,k}(\mathbb{N}\setminus A,n)\geq1.\end{equation}
It is clear that for any two different odds $j_1, j_2$ such that $j_1,j_2\in[0,\lfloor\frac{\lfloor\log_{k}{\frac{n}{T}}\rfloor}{2}\rfloor]$
and integers $i_1,i_2$ such that
$$i_1+j_1=K_1,~~~~i_2+j_2=K_2,$$
where $K_1,K_2\in\{\lfloor\log_{k}{\frac{n}{T}}\rfloor,\lfloor\log_{k}{\frac{n}{T}}\rfloor-1\},$ we have
\begin{equation}\label{e1}i_1\neq i_2.\end{equation}
In deed, assume that $j_1<j_2$, since $$1=-1+2\leq K_1-K_2-j_1+j_2=i_1-i_2,$$
it follows that $i_1\neq i_2$.
By \eqref{e1}, we have
\begin{equation}\label{equ2}[k^{i_1}t,k^{i_1}t+k^{i_1}-1]\cap[k^{i_2}t,k^{i_2}t+k^{i_2}-1]=\emptyset.\end{equation}
Therefore, by \eqref{equ1} and \eqref{equ2}, we have
$$R_{1,k}(A,n)=R_{1,k}(\mathbb{N}\setminus A,n)\geq\lfloor\frac{\lfloor\log_{k}{\frac{n}{T}}\rfloor}{4}\rfloor\gg \log n.$$
{\bf Case 2.} $(k^{i+j}+k^i-k-1)+1\leq r\leq k^i(k^j+1)-1$. Since $$|A\cap \{T,kT\}|=1,$$
it follows that
\begin{equation}\label{eqb2}|A(kT)|\geq1,~|(\mathbb{N}\setminus A)(kT)|\geq1.\end{equation}
Let $r=k^{i+j}+k^i-k-1+s,~s\in[1,k]$. Then
\begin{equation}\label{eqb12}n=k^i((k+1)t)+k^{i+j}+k^i-k-1+s=k^i((k+1)t+k^j)+k^i-k-1+s.\end{equation}
By \eqref{eq3}, we have
\begin{equation*}\label{eq37}
[k^i((k+1)t+k^j),k^i((k+1)t+k^j)+k^i-1]\subseteq A~\text{or}~\mathbb{N}\setminus A.
\end{equation*}
By \eqref{eqb2}, we can choose $a\in[0,kT]$ such that
\begin{equation}\label{eqb22}\{a\}\cup[k^i((k+1)t+k^j),k^i((k+1)t+k^j)+k^i-1]\subseteq A~\text{or}~\mathbb{N}\setminus A.\end{equation}
Since $j\in[0,\lfloor\frac{\lfloor\log_{k}{\frac{n}{T}}\rfloor}{2}\rfloor]$, it follows from
$$i+j=\lfloor\log_{k}{\frac{n}{T}}\rfloor~~\text{or}~~\lfloor\log_{k}{\frac{n}{T}}\rfloor-1$$
that
$$k^i-k-1\geq k^{\lfloor\frac{\lfloor\log_{k}\frac{n}{T}\rfloor}{2}\rfloor-1}-k-1\geq k^2T\geq ka$$
for any sufficiently larger $n$. It follows from \eqref{eqb12} and \eqref{eqb22} that
$$k^i((k+1)t+k^j)+s\leq n-ka\leq k^i((k+1)t+k^j)+k^i-k-1+s,$$
which implies that $n\in A+kA~\text{or}~(\mathbb{N}\setminus A)+k(\mathbb{N}\setminus A)$,
and so
$$R_{1,k}(A,n)=R_{1,k}(\mathbb{N}\setminus A,n)\geq1.$$
Up to now, we proved that for any given odd $j\in[0,\lfloor\frac{\lfloor\log_{k}{\frac{n}{T}}\rfloor}{2}\rfloor]$, we have
\begin{equation}\label{equ3}R_{1,k}(A,n)=R_{1,k}(\mathbb{N}\setminus A,n)\geq1.\end{equation}
By \eqref{e1}, we have
\begin{equation}\label{equ4}[k^{i_1}((k+1)t+k),k^{i_1}((k+1)t+k)+k^{i_1}-1]\cap[k^{i_2}((k+1)t+k),k^{i_2}((k+1)t+k)+k^{i_2}-1]=\emptyset.\end{equation}
Therefore, by \eqref{equ3} and \eqref{equ4}, we have
$$R_{1,k}(A,n)=R_{1,k}(\mathbb{N}\setminus A,n)\geq\lfloor\frac{\lfloor\log_{k}{\frac{n}{T}}\rfloor}{4}\rfloor\gg \log n.$$

This completes the proof of Theorem \ref{thm1}.

\end{document}